\documentclass[12pt]{amsart}

\usepackage{amssymb}
\usepackage{amsmath}
\usepackage{mathtools}
\usepackage{amsfonts}
\usepackage{bbm}
\usepackage{bm}
\usepackage{physics}
\usepackage{graphicx}
\usepackage{enumerate}

\usepackage[dvips, left=2.4cm, right=2.4cm , top=2.4cm, bottom=3.4cm]{geometry}

\usepackage{color}

\definecolor{red}{rgb}{1,0,0}

\definecolor{gre}{rgb}{0,0.7,0}

\definecolor{blu}{rgb}{0,0,1}

\usepackage{amsthm}

\newtheorem{theorem}{Theorem}[section]
\newtheorem{corollary}[theorem]{Corollary}
\newtheorem{lemma}[theorem]{Lemma}
\newtheorem{proposition}[theorem]{Proposition}

\theoremstyle{definition}

\theoremstyle{remark}

\newcommand{\e}[0]{\ensuremath{\mathrm{e}}}

\newcommand{\QQ}[0]{\ensuremath{\mathbb{Q}}}
\newcommand{\RR}[0]{\ensuremath{\mathbb{R}}}

\newcommand{\ZZ}[0]{\ensuremath{\mathbb{Z}}}
\newcommand{\ee}[0]{\ensuremath{\mathrm{e}}}
\newcommand{\ii}[0]{\ensuremath{\mathrm{i}}}
\newcommand{\SL}[0]{\ensuremath{\mathrm{SL}}}
\newcommand{\HH}[0]{\ensuremath{\mathbb{H}}}

\newcommand{\TT}[0]{\ensuremath{\mathbb{T}}}
\newcommand{\CC}[0]{\ensuremath{\mathbb{C}}}

\newcommand{\mcF}[0]{\ensuremath{\mathcal{F}}}

\newcommand{\ASL}[0]{\ensuremath{\mathrm{ASL}}}
\newcommand{\mcD}[0]{\ensuremath{\mathcal{D}}}
\newcommand{\GL}[0]{\ensuremath{\mathrm{GL}}}

\newcommand{\mcS}[0]{\ensuremath{\mathcal{S}}}

\begin{document}

\title[Intersections in $\mathrm{SL}(3, \mathbb{Z}) \backslash \mathrm{SL}(3, \mathbb{R})$ and roots of cubic congruences]{The distribution of intersections in $\mathrm{SL}(3, \mathbb{Z}) \backslash \mathrm{SL}(3, \mathbb{R})$ and lattices related to roots of cubic congruences}
\date{7 January 2026}
\author{Matthew Welsh}

\begin{abstract}
  In this note we study the distribution of the intersections between certain translates of closed orbits of the positive diagonal subgroup in $\SL(3, \mathbb{Z}) \backslash \SL(3, \mathbb{R})$ with a maximal parabolic subgroup.
  These intersections are closely connected to roots of congruences for certain monic, irreducible cubic polynomials $F(X) \in \mathbb{Z}[X]$.
  The main result is that the intersections, considered as a sequences in the diagonal subgroup and the parabolic subgroup, are jointly equidistributed.
  This implies that certain affine lattices determined by pairs of roots of the cubic congruences are jointly equidistributed with corresponding ideals in the associated ring of integers.
  We note that the techniques here roughly parallel those which have been developed to study the multidimensional Farey sequence, and one hopes that techniques to study roots of congruences will continue to develop.
\end{abstract}

\maketitle


\section{Introduction}
\label{sec:intro}

Let $G = \SL(3, \RR)$ and $\Gamma = \SL(3, \ZZ)$.
The positive diagonal subgroup $A \subset G$, which we parametrize by
\begin{equation}
  \label{eq:aa1}
  a(t) = 
  \begin{pmatrix}
    \ee^{-2t} & 0 & 0 \\
    0 & \ee^t & 0 \\
    0 & 0 & \ee^{t}
  \end{pmatrix}
  ,\ a_1(t_1) = 
  \begin{pmatrix}
    1 & 0 & 0 \\
    0 & \ee^{-t_1} & 0 \\
    0 & 0 & \ee^{t_1}
  \end{pmatrix}
\end{equation}
has closed orbits in $\Gamma \backslash G$ through special points $\Gamma g_l$ associated to ideal classes $[I_l]$ in totally real cubic orders, see eg Linnik's classic book \cite[Chapter~VII]{MR238801}, or, for a more modern treatment, works by Einsiedler, Lindenstrauss, Michel, and Venkatesh \cite{MR2515103, MR2776363}.
In this short note, our interest in these closed orbits of $A$ comes from their connection to roots of cubic congruences, solutions $\mu \bmod m$ to the congruence $F(\mu) \equiv 0 \bmod m$ with $F$ an integer, cubic polynomial.
Theorem \ref{theorem:rootsintersections} below connects these congruence-roots (under certain assumptions on $F$) with the points in the intersection
\begin{equation}
  \label{eq:intersection}
  \Gamma g_l A g_0^{-1} \cap \Gamma P,
\end{equation}
where $g_0$ is a basis matrix of the cubic order associated to the polynomial $F$ and
\begin{equation}
  \label{eq:Pdef}
  P = \{
  \begin{pmatrix}
    * & * & * \\
    0 & * & * \\
    0 & * & * 
  \end{pmatrix}
  \in G \}.
\end{equation}
We roughly order the points in (\ref{eq:intersection}) by considering
\begin{equation}
  \label{eq:intersection1}
  \Gamma g_l A g_0^{-1} \cap \bigcup_{t\leq T}\Gamma P_0 a(t)
\end{equation}
where
\begin{equation}
  \label{eq:P0def}
  P_0 = \{
  \begin{pmatrix}
    1 & * & * \\
    0 & * & * \\
    0 & * & * 
  \end{pmatrix}
  \in G \}.
\end{equation}
Our main theorem, theorem \ref{theorem:equidistribution} below, is that the coordinates of the points in (\ref{eq:intersection1}), considered as a sequence in $A \times P_0$ are jointly equidistributed as $T \to \infty$ with respect to the measures on $A$ and $P_0$ induced by Haar measure on $G$.

In terms of the congruence-roots, this implies that certain affine lattices related to pairs of roots become equidistributed in $\ASL(2, \ZZ) \backslash \ASL(2, \RR)$.
The proof of theorem \ref{theorem:equidistribution} provides a dynamical and geometric method for studying the statistical distribution of objects naturally related to roots of cubic congruences.
This marks progress in a long-standing goal to use techniques on $\Gamma \backslash G$ to studying roots of cubic congruences, see \cite{MR3496932}. 

The sequence of normalized roots of cubic congruences is known by work of Hooley  \cite{Hooley1964} to be equidistributed modulo one, however the methods used to show this have since been understood by work of Kowalski and Soundararajan \cite{KowalskiSoundararajan2020} to have little to do with the arithmetic of the polynomials themselves, but rather on properties of sequences that can be constructed via the Chinese remainder theorem from sets of congruence classes modulo primes.
In particular such methods cannot produce strong enough input for use in sieve methods such as in work on roots of quadratic congruences by Duke, Friedlander and Iwaniec \cite{DukeFriedlanderIwaniec1995}.

In other work, algebraic parametrizations of the roots of polynomial congruences have been used effectively to study somewhat related problems \cite{Heath-Brown2000, Dartyge2015, delaBreteche2015} that do not require strong results on the distribution of the congruence roots.
It therefore remains an attractive proposition to develop a geometric/dynamic understanding of the congruence roots in hopes that techniques from homogeneous dynamics or even the spectral theory of automorphic forms may be brought into the analysis of their distribution.

With this context, theorem \ref{theorem:equidistribution} draws interesting parallels to recent work on the two- and higher-dimensional Farey sequence, where both dynamic and analytic techniques have been successfully used to study their distribution.
We review some aspects of this work in the following section.

\subsection{Background on the Farey sequence}
\label{sec:background}

\subsubsection{$\SL(2)$ setting}
\label{sec:SL2}

Our methods and results are analogous to the study of the multidimensional Farey sequence, or equivalently, the rational points on closed horospheres.
The analogy is most clear in the $\SL(2)$ setting due to the identification of $\SL(2, \RR) / \{\pm I\}$ with the unit tangent bundle of the upper half plane $\HH$.
In Iwasawa coordinates, we have
\begin{equation}
  \label{eq:T1H}
  \pm
  \begin{pmatrix}
    1 & x \\
    0 & 1
  \end{pmatrix}
  \begin{pmatrix}
    y^{\frac{1}{2}} & 0 \\
    0 & y^{-\frac{1}{2}}
  \end{pmatrix}
  \begin{pmatrix}
    \cos \theta & -\sin \theta \\
    \sin \theta & \cos \theta
  \end{pmatrix}
  \mapsto ( x + \ii y, \ii y \ee^{-2\ii \theta}),
\end{equation}
where on the right we have written the unit tangent bundle of $\HH$ as the set of all pairs $(z, w) \in \CC^2$ with $z \in \HH$ and $|w| = \mathrm{Im}(z)$, so that $w$ has unit length in the hyperbolic metric.

Under this identification, the set of matrices of the form $
\begin{pmatrix}
  1 & * \\
  0 & 1
\end{pmatrix}
$
corresponds to a horizontal horocycle with tangent vectors pointing vertically upwards.
The image of this set under $
\begin{pmatrix}
  a & b \\
  c & d
\end{pmatrix}
\in \Gamma
$ is another horocycle, which will not be horizontal if $c \neq 0$.
In this case, the top of the horocycle is the point with tangent vector pointing vertically downwards.
The calculation
\begin{equation}
  \label{eq:horocycletop}
  \begin{pmatrix}
    a & b \\
    c & d
  \end{pmatrix}
  \begin{pmatrix}
    1 & -\frac{d}{c} \\
    0 & 1
  \end{pmatrix}
  =
  \begin{pmatrix}
    1 & \frac{a}{c} \\
    0 & 1
  \end{pmatrix}
  \begin{pmatrix}
    \frac{1}{c} & 0 \\
    0 & c
  \end{pmatrix}
  \begin{pmatrix}
    0 & -1 \\
    1 & 0
  \end{pmatrix}
\end{equation}
then shows that this top is $\tfrac{a}{c} + \tfrac{i}{c^2}$.

The sequence of Farey fractions $\tfrac{a}{c}$ may therefore be studied as the tops of the images of a horocycle under $\Gamma$ or as the points of intersection between the sets $\Gamma
\begin{pmatrix}
  1 & * \\
  0 & 1
\end{pmatrix}
$
and $\Gamma
\begin{pmatrix}
  * & * \\
  0 & *
\end{pmatrix}
\begin{pmatrix}
  0 & -1 \\
  1 & 0
\end{pmatrix}
$.
This latter perspective was successfully used by Athreya and Cheung \cite{MR3214280} to study gap distribution of Farey fractions and the BCZ map \cite{MR1837099}, and the method can be extended to lattices other than $\SL(2, \ZZ)$, see \cite{MR3330337, MR3567252} for example. 

If instead one starts with the geodesic in $\HH$ from $-\sqrt{2}$ to $\sqrt{2}$, then the tops of the images of this geodesic are points $\frac{\mu}{m} + \ii \frac{\sqrt{2}}{m}$ with $\mu^2 \equiv 2 \bmod m$.
This connection between the tops of the $\Gamma$-images of a geodesic and roots of the congruence is used to study the distribution of these roots in \cite{MR4654052}, \cite{MR4795423} and can be interpreted in terms of the intersections between
\begin{equation}
  \label{eq:geodesicintersection}
  \Gamma \frac{1}{(2\sqrt 2)^{\frac{1}{2}}}
  \begin{pmatrix}
    \sqrt{2} & -\sqrt{2} \\
    1 & 1
  \end{pmatrix}
  \begin{pmatrix}
    * & 0 \\
    0 & *
  \end{pmatrix}
  \mathrm{\ and\ }
  \Gamma
  \begin{pmatrix}
    * & * \\
    0 & *
  \end{pmatrix}
  \begin{pmatrix}
    \cos \frac{\pi}{4} & \pm \sin \frac{\pi}{4} \\
    \mp \sin \frac{\pi}{4} & \cos \frac{\pi}{4}
  \end{pmatrix}
  ,
\end{equation}
the latter corresponding to points with horizontal tangent vectors.

With some modifications to account for nontrivial narrow ideal classes of $\ZZ[\sqrt D]$, most significantly the use of multiple orbits of geodesics, this extends to the roots of the quadratic congruence $\mu^2 \equiv D \bmod m$ with $D > 1$ squarefree, see \cite{MR4654052,MR4795423}.
These roots are therefore naturally analogous to the Farey fractions with tops of horocycles being replaced with tops of geodesics, and they can be studied by similar techniques.
This analogy extends, in a somewhat simpler manner, to quadratic congruences with negative $D$ as well, but with geodesics/horocycles being replaced with $\mathrm{SO}(2)$ cosets of $\SL(2, \RR)$, ie points in $\HH$. 

\subsubsection{$SL(3)$ setting}
\label{sec:SL3}

The analogy requires more significant modification to extend in $\SL(3, \RR)$ to cubic congruences and the two-dimensional Farey sequence.
Instead of roots of cubic congruences themselves, our primary result concerns certain lattices associated to the roots, see theorem \ref{theorem:rootsintersections}.

The two-dimensional Farey fractions are pairs of rational numbers $(\frac{r_1}{q}, \frac{r_2}{q})$ with $q > 0$ and $\gcd(q, r_1, r_2) = 1$, ordered by the denominator $q$.
Considered modulo $\ZZ^2$, these points, like the classical 1-dimensional Farey sequence, are easily shown to equidistribute on the torus $\TT^2 = \RR^2 / \ZZ^2$.
Marklof \cite{MR3079137} studies the fine-scale distribution of these points and their higher-dimensional counterparts using dynamics on $\SL(3, \ZZ) \backslash \SL(3, \RR)$ and higher-rank homogeneous spaces.

A key observation in the method is that the points in the intersections
\begin{equation}
  \label{eq:closedhorosphere}
  \Gamma N \cap \bigcup_{t\leq T}\Gamma P_0 a(t),
\end{equation}
where
\begin{equation}
  \label{eq:Nbardef}
  N = \{ n(x_1,x_2) = 
  \begin{pmatrix}
    1 & 0 & 0 \\
    x_1 & 1 & 0 \\
    x_2 & 0 & 1 
  \end{pmatrix}
  : x_1,x_2 \in \RR\},
\end{equation}
occur at $n (\bm x) \in N$ with $\bm x$ a Farey fraction with denominator $q \leq \ee^{2T}$.
Moreover, the $\Gamma \backslash \Gamma P_0 \cong \ASL(2, \ZZ) \backslash \ASL(2, \RR)$ coordinates corresponding to the Farey fraction through this intersection provide information on the fine-scale distribution of the Farey points, see \cite{MR3079137}.
Limit theorems  on the fine-scale distribution follow from the joint equidistribution of the Farey points and these points in $\ASL(2, \ZZ) \backslash \ASL(2, \RR)$ as $T \to \infty$. 

Several subsequent works have studied these points using techniques from dynamics \cite{MR3484111} as well as Fourier-analytic methods \cite{MR3873538, elbaz2023effectiveequidistributionprimitiverational}.
In addition to studying the problem in any dimension, these works are able to prove this joint equidistribution for the points with a single denominator $q\to \infty$ as opposed to all denominators $q \leq \ee^{2T}$.

For the problem with an individual $q$, the relevant intersection (\ref{eq:closedhorosphere}), with a fixed $t$, is between two-dimensional and five-dimensional homogeneous subspaces of an eight-dimensional space.
These intersections are therefore considered ``unlikely'' and require additional arithmetic input to be understood fully.
Indeed, a key observation in \cite{ MR3873538, MR3484111, elbaz2023effectiveequidistributionprimitiverational} is that the $\SL(2, \ZZ) \backslash \SL(2, \RR)$ part of the $\ASL(2, \ZZ) \backslash \ASL(2, \RR)$ coordinates in the intersection (\ref{eq:closedhorosphere}) are closely related to the set of Hecke neighbors of $\ZZ^2$ with index $q$.
One can then appeal to results on the representation theory of Hecke operators such as \cite{MR2058609} to then prove their equidistribution.

It would be attractive if one could study the intersection (\ref{eq:intersection}) replaced by
\begin{equation}
  \label{eq:intersection2}
  \Gamma g_l A g_0^{-1} \cap \Gamma B,
\end{equation}
where $B$ is the subgroup of upper triangular matrices and we recall $\Gamma g_l$ correspond to ideal classes in a cubic order, since these would correspond to roots of cubic congruences directly, see theorem \ref{theorem:rootsintersections} and the remarks following below.
The intersections would be unlikely in this case as in \cite{MR3484111}, an intersection between a 2- and a 5-dimensional subspace inside of an 8-dimensional space, but so far no satisfactory replacement for \cite{MR2058609} has been found.
We do remark however that the $A$-coordinates of points in (\ref{eq:intersection2}) can be shown to equidistribute using analytic properties of L-functions associated to Hecke characters.

\subsection{Results}
\label{sec:results}

We recall
\begin{equation}
  \label{eq:aa1Adef}
  a(t) =
  \begin{pmatrix}
    \e^{-2t} & 0 & 0 \\
    0 & \e^t & 0 \\
    0 & 0 & \e^t
  \end{pmatrix}
  ,\ a_1(t_1) =
  \begin{pmatrix}
    1 & 0 & 0 \\
    0 & \e^{-t_1} & 0 \\
    0 & 0 & \e^{t_1}
  \end{pmatrix}
  ,\ A = \{ a(t)a_1(t_1) : t, t_1 \in \RR\},
\end{equation}
and
\begin{equation}
  \label{eq:Pdef1}
  P = \{
  \begin{pmatrix}
    * & * & * \\
    0 & * & * \\
    0 & * & * 
  \end{pmatrix}
  \in G \},\ P_0 = \{
  \begin{pmatrix}
    1 & * & * \\
    0 & * & * \\
    0 & * & * 
  \end{pmatrix}
  \in G\}. 
\end{equation}
We note that $P_0$ is isomorphic to $\ASL(2, \RR) = \SL(2, \RR) \ltimes \RR^2$ with multiplication
\begin{equation}
  \label{eq:aslmult}
  (g_1, v_1)(g_2, v_2) = (g_1g_2, v_1g_2 + v_2).
\end{equation}
We also note that $\Gamma \backslash \Gamma P_0$ is isomorphic to $\ASL(2, \ZZ) \backslash \ASL(2, \RR)$. 

We fix an irreducible polynomial $F(X) = X^3 + a_1 X^2 + a_2 X + a_3 \in \ZZ[X]$ and we let $\alpha =
\begin{pmatrix}
  \alpha^{(1)} & \alpha^{(2)} & \alpha^{(3)}
\end{pmatrix}
$
be a root of $F$, which we identify as an element of $\CC^3$ via the embeddings of $\QQ(\alpha)$ into $\CC$.
We note that mulciplication in $\QQ(\alpha)$ corresponds to componet-wise multiplication in $\CC$.

We assume that that $\alpha$ generates a totally real field, i.e. $\alpha \in \RR^3$, and we also assume that the ring $\ZZ[\alpha]$ is a maximal order, i.e. $\ZZ[\alpha]$ is the ring of integers in $\QQ(\alpha)$. 
We let $g_0 \in G$ be the basis matrix
\begin{equation}
  \label{eq:standardbasis}
  \begin{pmatrix}
    \alpha^2 \\
    \alpha \\
    1
  \end{pmatrix}
  =
  \begin{pmatrix}
    (\alpha^{(1)})^2 & (\alpha^{(2)})^2 & (\alpha^{(3)})^2 \\
    \alpha^{(1)} & \alpha^{(2)} & \alpha^{(3)} \\
    1 & 1 & 1
  \end{pmatrix}
\end{equation}
of the standard $\ZZ$-basis of $\ZZ[\alpha]$ rescaled and with embeddings ordered so that $\det g_0 = 1$.
In the same way, we let $g_1, \dots, g_h \in G$ be rescaled basis matrices of narrow ideal class representatives $I_1, \dots, I_h$.

To a totally positive $\xi \in \QQ(\alpha)$, we associate
\begin{equation}
  \label{eq:tildea}
  \tilde a(\xi) =
  \tfrac{1}{N(\xi)^{\frac{1}{3}}}
  \begin{pmatrix}
    \xi^{(1)} & 0 & 0 \\
    0 & \xi^{(2)} & 0 \\
    0 & 0 & \xi^{(3)}
  \end{pmatrix}
  ,
\end{equation}
where $N(\xi)$ is the norm of $\xi$.
In this way, the totally positive units $\mathcal{U}^+$ in $\ZZ[\alpha]$ can be identified with $ g_0^{-1} \Gamma g_0 \cap A =  g_l^{-1}\Gamma g_l \cap A$; note that since $\ZZ[\alpha]$ is the ring of integers, $u I = u$ implies that $u$ is a unit.
We have the following extension of theorem 6 of \cite{MR4467125}.

\begin{theorem}
  \label{theorem:rootsintersections}
  The points in the intersection of $\bigcup_{l=1}^h \Gamma g_l A g_0^{-1}$ with $\bigcup_{t\leq T}\Gamma P_0 a(t)$ correspond exactly to
  \begin{equation}
    \label{eq:intersectioncoordinates}
    \tilde{a}(\mathcal{U}^+ \xi) \in (g_l^{-1} \Gamma g_l \cap A) \backslash A
  \end{equation}
  and
  \begin{equation}
    \label{eq:intersectioncoordinates1}
    \Gamma 
    \begin{pmatrix}
      1 & \tfrac{\mu_1 + a_1}{m_1} & \tfrac{\lambda}{m_1m_2} \\
      0 & 1 & - \tfrac{\mu_2}{m_2} \\
      0 & 0 & 1
    \end{pmatrix}
    \begin{pmatrix}
      1 & 0 & 0 \\
      0 & m_2^{-\frac{1}{2}} & 0 \\
      0 & 0 & m_2^{\frac{1}{2}}
    \end{pmatrix}
    \in \Gamma \backslash \Gamma P_0
  \end{equation}
  with $N(\xi) = m_1^2 m_2 \leq \e^{6T}$, where $\xi \in I_l^{-1}$ is totally positive, $F(\mu_j) \equiv 0 \bmod {m_j}$, $\gcd(m_1, m_2, \mu_1 - \mu_2) = 1$, and where $\lambda$ is uniquely determined modulo $m_1m_2$ by $\mu_1$ and $\mu_2$, see (\ref{eq:idealcondition2}).
\end{theorem}

Our main result is the following.

\begin{theorem}
  \label{theorem:equidistribution}
  For each $l$, the points (\ref{eq:intersectioncoordinates}) and (\ref{eq:intersectioncoordinates1}) corresponding to intersections of $\Gamma g_l A g_0^{-1}$ and $\bigcup_{t\leq T}\Gamma P_0 a(t)$ are jointly equidistributed as $T \to \infty$ on $(g_l^{-1} \Gamma g_l \cap A) \backslash A$ and $\Gamma \backslash \Gamma P_0$ with respect to Haar measure.
\end{theorem}

Note that \eqref{eq:intersectioncoordinates1} in fact gives a point in $\Gamma \backslash \Gamma B_0 a_1(t_1)$ with $B_0$ the subgroup of upper triangular unipotent matrices.
Also note that one can recover $\tfrac{\mu_1 + a_1}{m_1} \bmod 1$ and $\tfrac{mu_2}{m_2} \bmod 1$ from the $\Gamma \cap B_0$-orbit of the matrix in \eqref{eq:intersectioncoordinates1}.
However, as the intersections between $\Gamma B$, $B$ the subgroup of upper triangular matrices with positive diagonal, and $\Gamma g_l A g_0^{-1}$ are unlikely in the sense that they are between five- and two-dimensional submanifolds of an eight-dimensional space, we are as yet unable to study these points with our methods.

In contrast, when considered as points in $\Gamma P_0$, where the intersections between $\Gamma P$ and $ g_l A g_0^{-1}$ are no longer unlikely, one is lead to the $\Gamma \cap P_0$-orbit of the matrix in \eqref{eq:intersectioncoordinates1}.
Due to this much larger group, it is unclear if one can obtain information on the distribution of the congruence-roots in $\RR/ \ZZ$ from these orbits.

\section{Proof of theorem \ref{theorem:rootsintersections}}
\label{sec:congruences}

From the proof of theorem 6 in \cite{MR4467125}, we see that a lattice contained in $\ZZ[\alpha]$ is an ideal if and only it has a basis $\{\beta_1, \beta_2, \beta_3\}$ of the form
\begin{equation}
  \label{eq:latticebasis}
  \begin{pmatrix}
    \beta_1 \\
    \beta_2 \\
    \beta_3
  \end{pmatrix}
  = a
  \begin{pmatrix}
    1 & \mu_1 + a_1 & \lambda \\
    0 & m_1 & - \mu_2 m_1 \\
    0 & 0 & m_1 m_2
  \end{pmatrix}
  \begin{pmatrix}
    \alpha^2 \\
    \alpha \\
    1
  \end{pmatrix}
  ,
\end{equation}
where $a \in \ZZ$, $F(\mu_j) \equiv 0 \bmod m_j$,
\begin{equation}
  \label{eq:idealcondition1}
  \mu_1^2 + \mu_1 \mu_2 + \mu_2^2 + a_1(\mu_1 + \mu_2) + a_2 \equiv 0 \pmod {\gcd(m_1, m_2)},
\end{equation}
and
\begin{multline}
  \label{eq:idealcondition2}
  \lambda \equiv (\mu_1^2 + a_1 \mu_1 + a_2) \frac{\overline{m_2}m_2}{\gcd(m_1, m_2)} - ( \mu_2^2 + \mu_1 \mu_2  + a_1 \mu_2) \frac{\overline{m_1}m_1}{\gcd(m_1, m_2)} \\
  + \kappa \frac{m_1m_2}{\gcd(m_1, m_2)} \pmod { m_1m_2},
\end{multline}
where
\begin{equation}
  \label{eq:overlinemdef}
  \frac{\overline{m_1}m_1}{\gcd(m_1, m_2)} + \frac{\overline{m_2}m_2}{\gcd(m_1, m_2)} = 1
\end{equation}
and $\kappa$ is a solution to
\begin{equation}
  \label{eq:kappadef}
  (\mu_2 - \mu_1)\kappa \equiv \frac{F(\mu_1)}{m_1} \overline{m_2} + \frac{F(\mu_2)}{m_2} \overline{m_1} \pmod {\gcd(m_1, m_2)}. 
\end{equation}
For example, if $p$ is a prime and $F(\mu) \equiv 0 \bmod p$, the lattices associated to the following matrices are ideals:
\begin{equation}
  \label{eq:idealexamples}
  \begin{pmatrix}
    1 & 0 & -\mu^2 \\
    0 & 1 & -\mu \\
    0 & 0 & p
  \end{pmatrix}
  \quad \mathrm{and}\quad
  \begin{pmatrix}
    1 & \mu + a_1 & \mu^2 + a_1 \mu + a_2 \\
    0 & p & 0 \\
    0 & 0 & p
  \end{pmatrix}
  .
\end{equation}

To prove theorem \ref{theorem:rootsintersections}, we begin with the following lemma, which shows that certain congruences cannot be satisfied under the assumption that all ideals are invertible, i.e. that $\ZZ[\alpha]$ is a maximal order.
\begin{lemma}
  \label{lemma:noninvertibility}
  Suppose that a prime $p$ and $\mu \bmod {p^2}$ satisfy $F(\mu) \equiv 0 \bmod {p^2}$ and $F'(\mu) \equiv 0 \bmod p$.
  Then we have $J_1 J_2 = p J_1$ where $J_1$ and $J_2$ are the ideals corresponding to the left and right matrices in \eqref{eq:idealexamples}, respectively. 
\end{lemma}

\begin{proof}
  We let $\beta_1, \beta_2, \beta_3$ and $\gamma_1, \gamma_2, \gamma_3$ denote the basis elements of $J_1$ and $J_2$ according to the matrices (\ref{eq:idealexamples}).
  We have that
  \begin{equation}
    \label{eq:productbasis}
    \delta_1 := \beta_1 \gamma_3 = p\alpha^2 - p \mu^2,\ \delta_2 := \beta_2 \gamma_3 = p\alpha - p\mu,\ \delta_3 := \beta_3 \gamma_3 = p^2
  \end{equation}
  forms a basis for $p J_1$.
  Now
  \begin{equation}
    \label{eq:productideal}
    \begin{split}
      \beta_1 \gamma_2 & = p\alpha^3 - p\mu^2 \alpha = -a_1 \delta_1 - (a_2 + \mu^2)\delta_2 -\frac{F(\mu)}{p} \delta_3, \\
      \beta_2 \gamma_2 & = p \alpha^2 - p\mu \alpha = p \delta_1 - \mu \delta_2,\\
      \beta_3 \gamma_2 & = p^2 \alpha = p \delta_2 + \mu \delta_3,\\
      \beta_1 \gamma_1 & = -\mu F(\mu) - F(\mu) \alpha = -\frac{F(\mu)}{p} \delta_2 + 2 \frac{F(\mu)}{p^2} \delta_3, \\
      \beta_2 \gamma_1 & = - F(\mu) = - \frac{F(\mu)}{p^2} \delta_3, \\
      \beta_3 \gamma_1 & = p \alpha^2 + p(\mu + a_1) \alpha + p(\mu^2 + a_1 \mu + a_2) = \delta_1 + (\mu + a_1) \delta_2 + \frac{F'(\mu)}{p} \delta_3. 
    \end{split}
  \end{equation}
  These calculations show that $J_1J_2$ is generated by $\delta_1, \delta_2, \delta_3$, so $J_1 J_2 = p J_1$ as claimed. 
\end{proof}

\begin{proof}[Proof of theorem \ref{theorem:rootsintersections} for $\Gamma = \SL(3, \ZZ)$]
  The basis for a lattice $I\subset \ZZ[\alpha]$ given in (\ref{eq:latticebasis}) is unique up to left multiplication by unipotent upper triangular matrices in $\Gamma$.
  We note that if $\gcd(m_1, m_2, \mu_1 - \mu_2) = 1$, then $F(\mu_1) \equiv F(\mu_2) \equiv 0 \pmod{\gcd(m_1, m_2)}$ implies
  \begin{equation}
    \label{eq:disjoint}
    \mu_1^2 + \mu_1 \mu_2 + \mu_2^2 + a_1(\mu_1 + \mu_2) + a_2 = \frac{F(\mu_1) - F(\mu_2)}{\mu_1 - \mu_2} \equiv 0 \pmod {\gcd(m_1, m_2)},
  \end{equation}
  so the condition (\ref{eq:idealcondition1}) is satisfied.
  The equation (\ref{eq:kappadef}) can also be solved for a unique $\kappa$, so the condition $\gcd(m_1, m_2, \mu_1 - \mu_2) = 1$ implies $I$ is an ideal.

  Now suppose $I$ is an ideal and assume that a prime $p$ divides $\gcd(m_1, m_2, \mu_1 - \mu_2)$.
  The condition (\ref{eq:idealcondition1}) implies $F'(\mu) \equiv 0 \bmod p$ and (\ref{eq:kappadef}) having a solution implies $F(\mu_j) \equiv 0 \bmod p^2$ for either $j = 1$ or $j = 2$ as we can take $\overline{m_j} \equiv 1 \bmod p$ and the other $\overline{m_i} \equiv 0 \bmod p$.
  Using lemma \ref{lemma:noninvertibility} we then obtain a contradiction since $\ZZ[\alpha]$ is a maximal order. 
  We conclude that the condition $\gcd(m_1, m_2, \mu_1 - \mu_2) = 1$ is necessary and sufficient for $I$ to be an ideal.

  Every ideal $I$ is $\xi I_l$ for some totally positive $\xi \in I_l^{-1}$ that is unique up to multiplication by the totally positive units of $\ZZ[\alpha]$.
  Fixing $\ZZ$-bases $g_l$ for the narrow ideal class representatives $I_l$ as above, we obtain a basis for $I$ via $g_l \tilde{a}(\xi)$.
  There is some $\gamma \in \Gamma$ so that
  \begin{equation}
    \label{eq:changebasis}
    \gamma g_l \tilde{a}(\xi) = (m_1^2 m_2)^{-\frac{1}{3}}
    \begin{pmatrix}
      1 & \mu_1 + a_1 & \lambda \\
      0 & m_1 & - \mu_2 m_1 \\
      0 & 0 & m_1 m_2
    \end{pmatrix}
    g_0.
  \end{equation}
  This shows that the points described in theorem \ref{theorem:rootsintersections} are indeed in the intersection of $\Gamma g_l A g_o^{-1} \cap \bigcup_{t \leq T} \Gamma P_0 a(t)$.
  
  To show that the intersection contains no other points, we fix $\ZZ$-bases $\{\overline{\beta}_1, \overline{\beta}_2, \overline{\beta}_3\}$ for $I_l^{-1}$ and define the three $3\times 3$ integer matrices $B_k = (b_{ijk})$ by
  \begin{equation}
    \label{eq:Bkdef}
    \overline{\beta}_k\beta_i  = b_{i1k} \alpha^2 + b_{i2k} \alpha + b_{i3k} \textrm{ so that }
    \begin{pmatrix}
      \overline{\beta}_k\beta_1 \\
      \overline{\beta}_k\beta_2 \\
      \overline{\beta}_k\beta_3
    \end{pmatrix}
    = B_k g_0,
  \end{equation}
  where $\{\beta_1, \beta_2, \beta_3\}$ is a $\ZZ$-basis for $I_l$.
  We let $B$ be the matrix with $k$th column the first column of $B_k$, then from lemma 13 of \cite{MR4467125} we have that $B \in \GL(3, \ZZ)$.

  Let $\xi = c_1 \overline{\beta}_1 + c_2 \overline{\beta}_2 + c_3 \overline{\beta}_3$, so that $\xi \in I_l^{-1}$ if and only if all the $c_i$ are integers.
  We have
  \begin{equation}
    \label{eq:xiintegral}
    g_l \tilde{a}(\xi) =
    (N(I_l)N(\xi))^{-\frac{1}{3}}
    \begin{pmatrix}
      \xi\beta_1 \\
      \xi\beta_2 \\
      \xi\beta_3
    \end{pmatrix}
    = (N(I_l)N(\xi))^{-\frac{1}{3}} (c_1 B_1 + c_2 B_2 + c_2 B_3) g_0,
  \end{equation}
  so if there is $\gamma \in \Gamma$ so that
  \begin{equation}
    \label{eq:integrality}
    \gamma g_l \tilde{a}(\xi) =
    \begin{pmatrix}
      * & * & * \\
      0 & * & * \\
      0 & * & * 
    \end{pmatrix}
    g_0,
  \end{equation}
  then
  \begin{equation}
    \label{eq:cieq}
    B
    \begin{pmatrix}
      c_1 \\
      c_2 \\
      c_3
    \end{pmatrix}
    = \gamma^{-1}
    \begin{pmatrix}
      * \\
      0 \\
      0 
    \end{pmatrix}
    . 
  \end{equation}
  It follows that $(c_1, c_2, c_3)$ is a multiple of an integer vector, so after scaling we have $\xi \in I_l^{-1}$.
  That the intersection $\Gamma g_l A g_0^{-1} \cap \Gamma P$ only has points as in the right side of (\ref{eq:changebasis}) now follows as above from the fact that $I = \xi I_l$ is an ideal contained in $\ZZ[\alpha]$. 
\end{proof}

\section{Proof of theorem \ref{theorem:equidistribution}}
\label{sec:equidistribution}

\subsection{Fundamental domain for the totally positive units}
\label{sec:units}

In this section we construct a simple fundamental domain $\mcD \subset \RR_{>0}^3$ for the action of the totally positive units on $\RR_{>0}^3$, meaning for every $\xi \in \RR_{>0}^3$, there is exactly one totally positive unit $u$ such that $\xi u \in \mcD$.
Via the embeddings, $\mcD$ also serves as a fundamental domain for the action of the totally positive units on the totally positive elements of $\QQ(\alpha)$.

By Dirichlet's unit theorem, the map from $\RR_{>0}^3 \to \RR^3$ given by $\ell : (x_1, x_2, x_3) \mapsto (\log x_1,\allowbreak \log x_2,\allowbreak \log x_3)$ sends the totally positive units to a rank two lattice contained in the plane
\begin{equation}
  \label{eq:unitplane}
  \{
  \begin{pmatrix}
    x_1 & x_2 & x_3
  \end{pmatrix}
  \in \RR^3 : x_1 + x_2 + x_3 = 0 \}. 
\end{equation}
We fix generators (eg, we may pick the reduced basis) of this lattice, $\ell(\varepsilon_1)$ and $\ell(\varepsilon_2)$, say, and we define $\mcD$ by
\begin{equation}
  \label{eq:mcDdef}
  \mcD  = \{ x \in \RR^3_{> 0} : \ell(x) = s_1 \ell (\varepsilon_1) + s_2 \ell(\varepsilon_2) \textrm{ with } s_1,s_2 \in (0,1]^2\}. 
\end{equation}

Let $\tilde \mcD$ be the intersection of $\mcD$ with the surface of $x \in \RR^3_{> 0}$ with $x_1x_2x_3 = 1$.
In what follows, we need the following geometric observation about $\tilde\mcD$. 
\begin{lemma}
  \label{lemma:mcDgeometry}
  Any plane through the origin that intersects $\tilde\mcD$ does so transversely in a curve of bounded length.
\end{lemma}

\begin{proof}
  We observe that since $\tilde\mcD$ is contained in the first quadrant, the normal of any intersecting plane must both positive and negative components.
  The normal vectors to $\tilde\mcD$ have positive components, so cannot be parallel.

  The claim on the length of the intersection follows from the compactness of $\tilde\mcD$. 
\end{proof}

\subsection{Equidistribution of expanding flats}
\label{sec:expandingflats}

We recall that
\begin{equation}
  \label{eq:Ndef}
  N = \{ n(x_1, x_2) =
  \begin{pmatrix}
    1 & 0 & 0 \\
    x_1 & 1 & 0 \\
    x_2 & 0 & 1
  \end{pmatrix}
  : x_1, x_2 \in \RR\}.
\end{equation}
As $\Gamma N a(-T)$ is the expanding horosphere as $T\to\infty$, it is easily seen to equidistribute in $\Gamma \backslash G$ from Margulis thickening, see \cite{MR2035655}, \cite{MR1359098}; theorem 5.3 in \cite{MR2726104} states the following.
\begin{theorem}
  \label{theorem:horosphereequidistribution}
  For any continuous, bounded function $f$ on $\TT^2 \times \Gamma \backslash G$,
  \begin{equation}
    \label{eq:effectiveequidistribution}
    \lim_{T \to \infty} \int_{\TT^2} f(\bm x, \Gamma n(\bm x) a(-T)) \dd \bm x = \int _{\TT^2} \int_{\Gamma \backslash G} f(\bm x,  g) \dd \mu(g) \dd \bm x,
  \end{equation}
  where $\mu $ is Haar measure on $G$ normalized to be a probability measure on $\Gamma \backslash G$. 
\end{theorem}

Using theorem \ref{theorem:horosphereequidistribution} we prove corollary \ref{corollary:expandingflats} on the equidistribution of translates of the closed orbit $\Gamma g_l A$ by $g_0 a(-T)$.
This is almost an immediate consequence of theorem \ref{theorem:horosphereequidistribution} since the set $g_l A g_0^{-1}$ projects to the full unstable manifold of $a_1(-T)$.
We first estimate the measure of the subset of $g_l A g_0^{-1}$ on which the LU decomposition is badly behaved.
\begin{lemma}
  \label{lemma:badLU}
  For any $\epsilon >0$, let $\mathcal{B}(\epsilon)$ be the set of $a \in \tilde\mcD$ such that the (1,1)-entry of $g_l a g_0^{-1}$ has absolute value at most $\epsilon$.
  Then $\mathcal{B}(\epsilon) $ has measure $O(\epsilon)$ with respect to Haar measure on $A$ and the boundary of $\mathcal{B}(\epsilon)$ has length $O(1)$.
\end{lemma}

\begin{proof}
  The $(1,1)$-entry of the matrix $g_l \tilde a(\xi) g_0^{-1}$ is given by the linear functional
  \begin{equation}
    \label{eq:linearfunction}
    \begin{pmatrix}
      \xi^{(1)} \\
      \xi^{(2)} \\
      \xi^{(3)}
    \end{pmatrix}
    \mapsto
    \begin{pmatrix}
      1 & 0 & 0 
    \end{pmatrix}
    g_l
    \begin{pmatrix}
      \xi^{(1)} & 0 & 0 \\
      0 & \xi^{(2)} & 0 \\
      0 & 0 & \xi^{(3)}
    \end{pmatrix}
    g_0^{-1}
    \begin{pmatrix}
      1 \\
      0 \\
      0
    \end{pmatrix}
    .
  \end{equation}
  By lemma \ref{lemma:mcDgeometry} the kernel of this functional intersects $\tilde\mcD$ transversely in a curve of bounded length.
  Moreover, the gradient of the linear functional is bounded, depending only on $g_l$ and $g_0$, so the lemma follows. 
\end{proof}

Define $\hat a : \TT^2 \to (g_l^{-1} \Gamma g_l \cap  A) \backslash A$ by $\hat a (s_1, s_2) = \tilde a(\varepsilon_1^{s_1}\varepsilon_2^{s_2})$.
We have the following consequence of theorem \ref{theorem:horosphereequidistribution}.

\begin{corollary}
  \label{corollary:expandingflats}
  For any continuous, bounded function $f$ on $\TT^2 \times \Gamma \backslash G$,
  \begin{equation}
    \label{eq:expandingflats}
    \lim_{T \to \infty} \int_{(g_l \Gamma g_l^{-1} \cap A) \backslash A} f( \bm s, g_l \hat a(\bm s) g_0^{-1} a(-T)) \dd \bm s = \int_{\TT^2} \int_{\Gamma \backslash G} f(\bm s, g) \dd \mu(g) \dd a.
  \end{equation}
\end{corollary}

\begin{proof}
  We identify $(g_l \Gamma g_l^{-1}\cap A) \backslash A $ with $\tilde \mcD$ and we let $\mathcal{S}(\epsilon)$ be the set of $\bm s$ corresponding to $a \in \mathcal{B}(\epsilon)$.
  For $\bm s\not\in \mcS(\epsilon)$, we have the LU decomposition
  \begin{equation}
    \label{eq:LUdecomp}
    g_l \hat a (\bm s) g_0^{-1} a(-T)= n(\bm x (\bm s)) a(-T)a( r(\bm s))
    \begin{pmatrix}
      \pm 1 & \ee^{-3T} \bm u (\bm s) \\
      0 & V(\bm s)
    \end{pmatrix}
    ,
  \end{equation}
  for functions $r$, $\bm u$, $V$ that are uniformly continuous on $\TT^2 \setminus \mcS(\epsilon)$.
  From this continuity, for any $\delta_1 > 0$, there is $\delta_2 > 0$ (depending also on $\epsilon$) so that on any square $S \subset \TT^2 \setminus \mcS(\epsilon)$ with size $\delta_2$ there are $g_S \in G$ for which every $\bm s \in S$ satisfies
  \begin{equation}
    \label{eq:stableremove}
    n(\bm x (\bm s)) a(-T) a(r(\bm s))
    \begin{pmatrix}
      \pm 1 & \ee^{-3T} \bm u (\bm s) \\
      0 & V(\bm s)
    \end{pmatrix}
    = n(\bm x (\bm s)) a(-T) g_S(I + O(\delta_1)).
  \end{equation}
  By lemma \ref{lemma:badLU}, we can cover a subset of $\TT^2 \setminus \mcS(\epsilon)$ of measure $1 +O(\epsilon + \delta_2)$ with such squares $S$.
  
  From the continuity of $f$ with respect to the left invariant metric, it follows that we can choose $\delta_1 = \delta_1(\epsilon)$ small enough so that for $\bm s \in S$,
  \begin{equation}
    \label{eq:frewrite}
    f(\bm s, g_l \hat a(\bm s) g_0^{-1} a(-T)) = f(\bm s, n(\bm x (\bm s)) a(-T) g_S) + O(\epsilon). 
  \end{equation}
  We may now approximate the left hand side of (\ref{eq:expandingflats}) up to $O(\epsilon)$ by
  \begin{equation}
    \label{eq:freplace}
    \sum_S \int_S f_S(n(\bm x (\bm s)) a(-T)) \dd \bm s,
  \end{equation}
  where
  \begin{equation}
    \label{eq:f_Sdef}
    f_S( g) = f( \bm s_S, g g_S)
  \end{equation}
  for some choice of fixed points $\bm s_S \in S$.
  By theorem \ref{theorem:horosphereequidistribution}, we have for all $T$ sufficiently large, 
  \begin{equation}
    \label{eq:f_Sdistribution}
    \int_S f_S(\overline{n}(\bm x (\bm s)) a(-T)) \dd \bm s = |S| \int_{\Gamma \backslash G} f_S(g) \dd \mu(g) + O(\epsilon \delta_2^2);
  \end{equation}
  note that the map $\bm s \mapsto \bm x( \bm s)$ has continuously differentiable inverse away from $\mathcal{S}(\epsilon)$. 

  We now have
  \begin{equation}
    \label{eq:histogram}
    |S| f_S(g) = |S| f(\bm s_S, gg_S) = \int_S f(\bm s, gg_S) \dd s + O(\epsilon \delta_2^2)
  \end{equation}
  by taking $\delta_2$ smaller if necessary.
  Finally (\ref{eq:expandingflats}) follows by changing variables to remove $g_S$ and then summing over the squares $S$, recalling that they cover a subset of $\TT^2$ having complementary measure $O(\epsilon + \delta_2)$.
\end{proof}

\subsection{Surface of section}
\label{sec:section}

To detect the points in the intersection $\Gamma g_l A g_0^{-1} \cap \bigcup_{t\leq T}\Gamma P_0 a(t)$, we apply corollary \ref{corollary:expandingflats} with a test function $f$ coming from a slight thickening of a truncation of the surface $\bigcup_{t\leq 0}\Gamma P_0a(t)$.
Specifically, we let $\mathcal{F}_1(Y)$ denote the part of the standard fundamental domain for $\SL(2, \ZZ) \backslash \SL(2, \RR)$ with height at most $Y$, identified with the semisimple factor of the parabolic subgroup $P$, i.e. $g \in \SL(2, \RR) \mapsto 
\begin{pmatrix}
  1 & \\
  & g
\end{pmatrix}
\in P$.
We set
\begin{equation}
  \label{eq:Sdef}
  S = S(Y) = \{ \Gamma \bar n a(t) m : \bar n \in \bar N, -T_0 \leq t \leq 0, m \in \mcF_1(Y)\},
\end{equation}
where
\begin{equation}
  \label{eq:barNdef}
  \bar N = \{
  \begin{pmatrix}
    1 & x_1 & x_2 \\
    0 & 1 & 0 \\
    0 & 0 & 1
  \end{pmatrix}
  : x_1, x_2 \in \RR\}. 
\end{equation}
We then thicken $S$ to
\begin{equation}
  \label{eq:Sepsilondef}
  S_\epsilon = S_\epsilon (Y) = \{ \Gamma p n(\bm{x}) : \Gamma p \in S, |\bm x| \leq \epsilon \},
\end{equation}
where $\epsilon = \epsilon(T_0, Y) > 0$ is chosen small enough so that
\begin{equation}
  \label{eq:epsilonchoice}
  \Gamma \cap \{ p_1 n(\bm x_1) n(\bm x_2)^{-1}  p_2^{-1}: \Gamma p_j \in S, |\bm x_j | \leq \epsilon \} \subset \Gamma \cap P.
\end{equation}
This can be done as $\Gamma $ is discrete and the set in (\ref{eq:epsilonchoice}) is compact and, for any point $g \not\in P$, $g$ is not in this set for all sufficiently small $\epsilon$. 
We note that the containment (\ref{eq:epsilonchoice}) implies that $S_\epsilon$ does not self-intersect.

\begin{lemma}
  \label{lemma:selfintersections}
  The map $(\bar n, t, m, \bm x) \mapsto S_\epsilon$ gives smooth coordinates on $S_\epsilon$, so in particular the surface $S$ is smoothly embedded in $\Gamma \backslash G$ and its thickening $S_\epsilon$ has no self-intersections.
  Moreover, the Haar measure on $G$ in these coordinates is given by
  \begin{equation}
    \label{eq:Haardecomp}
    \dd \mu(\bar n a(t) m n) = \ee^{6t} \dd \bar n \dd t \dd m \dd n.
  \end{equation}
\end{lemma}

\begin{proof}
  We first observe that $(\bar n, t, m, \bm x) \mapsto \bar n a(t) m n(\bm x)$ gives smooth coordinates on the set of $g \in G$ with $g^{-1}$ having nonzero $(1,1)$-entry.
  The first part of the lemma then follows from (\ref{eq:epsilonchoice}).
  
  From proposition 8.45 of \cite{MR1920389}, we have that the Haar measure on $g = \bar n a m n$ is given by
  \begin{equation}
    \label{eq:Haardecomp1}
    \dd \mu(g) = \ee^{2 \rho_A \log a} \dd \bar n \dd a \dd m \dd n,
  \end{equation}
  where $2 \rho_A$ is the sum of positive roots for the choice of $A$.
  In this case this is $\bm e_1 - \bm e_2 + \bm e_1 - \bm e_3$, so $2 \rho_A \log a(t) = 6t$.
  Equation (\ref{eq:Haardecomp}) follows. 
\end{proof}

From lemma \ref{lemma:selfintersections}, for any $\delta > 0$, we may define continuous functions $\chi^\pm : \Gamma \backslash G$ so that
\begin{equation}
  \label{eq:chidef}
  0\leq \chi^-(g) \leq \mathbbm{1}_{S_\epsilon}(g) \leq \chi^+(g),\  \mathrm{for\ all\ } g \in \Gamma \backslash G,
\end{equation}
and for each $g \in S$,
\begin{equation}
  \label{eq:chidef1}
  \int_{|\bm x | \leq \epsilon} \chi^+(g\overline{n}(\bm x)) \dd \bm x - \delta \leq \mathrm{vol}( \{ \bm x \in \RR^2: |x| \leq \epsilon\})  \leq \int_{| \bm x| \leq \epsilon} \chi^-(g\overline{n}(\bm x)) \dd \bm x + \delta.
\end{equation}

Applying corollary \ref{corollary:expandingflats} to $\chi^\pm$ leads to the following proposition.
\begin{proposition}
  \label{prop:truncatedequidistribution}
  For any bounded, continuous function $f$ on $(g_l^{-1} \Gamma g_l \cap A) \backslash A \times \Gamma \backslash G $, we have
  \begin{multline}
    \label{eq:truncatedequidistribution}
    \lim_{T \to \infty}  \sum_{\substack{ \xi \in I_l^{-1} \cap \mcD \\ \ee^{6(T - T_0)} < N(\xi) \leq \ee^{6T}  }} \ee^{-6T} f( \tilde a(\xi), n(\frac{\mu_1 + a_1}{m_1}, \frac{\lambda}{m_1m_2}) a(-T) n_1( -\frac{\mu_2}{m_2}) a_1(\tfrac{1}{2} \log m_2)) \\
    \times \mathbbm{1}_{\mcF_1(Y)}(n_1( -\frac{\mu_2}{m_2}) a_1(\tfrac{1}{2} \log m_2)) 
    = \int_{(g_l^{-1}\Gamma g_l \cap A) \backslash A}\int_{S}  f(a, g) \dd \nu(g) \dd a,
  \end{multline}
  where $\nu$ is the restriction of $\mu$ to $P$ and $\mu_j, m_j, \lambda$ are functions of $\xi$ as in theorem \ref{theorem:rootsintersections}. 
\end{proposition}

In (\ref{eq:truncatedequidistribution}) we have set $
  n_1(x) =
  \begin{pmatrix}
    1 & 0 & 0 \\
    0 & 1 & x \\
    0 & 0 & 1
  \end{pmatrix}
  $.

\begin{proof}
  By breaking into positive and negative parts, we may assume that $f$ is nonnegative.
  Moreover, we may choose $\epsilon$ small enough so that for any $g \in S$ and $a \in A$,
  \begin{equation}
    \label{eq:continuousf}
    |f(a, gn (\bm x)) - f(a, g) | \leq \delta
  \end{equation}
  for any $|\bm x | \leq \epsilon$.
  
  We now note that by theorem \ref{theorem:rootsintersections}, the second argument of $f$ on the left side of (\ref{eq:truncatedequidistribution}) can be written as $g_l a(\xi) g_0^{-1} a(-T)$. 
  From the continuity of $f$, we have that the sum in (\ref{eq:truncatedequidistribution}) is well approximated by
  \begin{equation}
    \label{eq:integralapprox}
    \frac{1}{\mathrm{vol}( \{ \bm x \in \RR^2: |x| \leq \epsilon\})} \int_{(g_l^{-1} \Gamma g_l \cap  A)  \backslash A} f(a, g_l a g_0^{-1} ) \mathbbm{1}_{S_\epsilon}(g_l a g_0^{-1}a(-T)) \dd a,
  \end{equation}
  using the fact that the small region of $a(-T) n(\bm x) a(T) = \overline{n}( \ee^{-3T} \bm x)$ with $|\bm x| \leq \epsilon$ has area $ \ee^{-6T} \mathrm{vol}( \{ \bm x \in \RR^2: |x| \leq \epsilon\})$.

  The integral (\ref{eq:integralapprox}) in turn can be bounded above and below by
  \begin{equation}
    \label{eq:chiintegral}
    \frac{1}{\mathrm{vol}( \{ \bm x \in \RR^2: |x| \leq \epsilon\})} \int_{(g_l^{-1} \Gamma g_0 \cap  A)  \backslash A} f(a, g_l a g_o^{-1}a(-T)) \chi^\pm ( g_l a g_0^{-1} a(-T) ) \dd a.
  \end{equation}
  Corollary \ref{corollary:expandingflats} implies that the limit as $T \to\infty$ of (\ref{eq:chiintegral}) equals
  \begin{equation}
    \label{eq:limit}
    \frac{1}{\mathrm{vol}( \{ \bm x \in \RR^2: |x| \leq \epsilon\})} \int_{(g_l^{-1} \Gamma g_0 \cap  A)  \backslash A} \int_{\Gamma \backslash G} f(a, g) \chi^\pm(g) \dd\mu(g) \dd a
  \end{equation}
  Using (\ref{eq:continuousf}), (\ref{eq:chidef1}) and the expression (\ref{eq:Haardecomp}) for $\dd \mu$, in the limit as $\epsilon, \delta \to 0$, (\ref{eq:limit}) becomes 
  \begin{equation}
    \frac{1}{c} \int_{(g_l^{-1} \Gamma g_0 \cap  A)  \backslash A} \int_N \int_{-T_0}^0 \int_{\mcF_1(Y)}  f(a, \bar n a(t)m) \ee^{6t} \dd \bar n \dd t \dd m,
  \end{equation}
  where $c > 0$ is the normalization for the total measure of $\tfrac{1}{c}\dd \bar n \dd t \dd n \dd m$ over $\Gamma \backslash G$ to be $1$. 
  This is the right side of (\ref{eq:truncatedequidistribution}) with the measure $\nu$ on $S$ given by
  \begin{equation}
    \label{eq:nudef}
    \dd \nu( \bar n a(t) m) = \tfrac{1}{c}\ee^{6t} \dd \bar n \dd t \dd m.
  \end{equation}
\end{proof}

\subsection{Non-divergence}
\label{sec:nondivergence}

To extend proposition \ref{prop:truncatedequidistribution} to theorem \ref{theorem:equidistribution}, it remains to remove the restriction $N(\xi) \geq \ee^{T-T_0}$ and $n_1(-\frac{mu_2}{m_2})a_1(\tfrac{1}{2} \log m_2) \in \mcF_1 (Y)$ from the left hand side of (\ref{eq:truncatedequidistribution}) and remove the corresponding restrictions on $S$ on the right hand side of (\ref{eq:truncatedequidistribution}).
These latter restrictions are easily removed using the expression for the measure $\nu$: the measure of $m \in \SL(2, \ZZ) \backslash \SL(2, \RR)$ with height greater than $Y$ has measure $O(Y^{-1})$ and the measure of $t \leq -T_0$ is $O(\ee^{-6T_0})$.

We remove the restriction that $n_1(-\frac{mu_2}{m_2})a_1(\tfrac{1}{2} \log m_2)\in \mcF (Y)$ using the following lemmas.
Here we let $\mcF(Y)$ denote the set of $\Gamma g \in \Gamma \backslash G$ such that
\begin{equation}
  \label{eq:minbottomrow}
  \min_{\gamma \in \Gamma} |
  \begin{pmatrix}
    0 & 0 & 1
  \end{pmatrix}
  g | \geq Y^{-\frac{1}{2}}.
\end{equation}
\begin{lemma}
  \label{lemma:smalltobigcusp}
  Every $n_1(-\frac{\mu_2}{m_2})a_1(\tfrac{1}{2} \log m_2) \not\in \mcF_1 (Y)$ with $\ee^{6(T - T_0)} < m_1^2m_2 \leq \ee^{6T}$, corresponds to a totally positive $\xi \in I_l^{-1}$ and satisfying $g_l a(\xi) g_0^{-1} a(-T ) \not\in \mcF(Y \ee^{-4T_0})$
\end{lemma}

\begin{proof}
  This follows immediately from the observation that if $
  \begin{pmatrix}
    a & b \\
    c & d
  \end{pmatrix}
  $
  are so that the bottom row of
  \begin{equation}
    \label{eq:cusp}
    \begin{pmatrix}
      a & b \\
      c & d
    \end{pmatrix}
    \begin{pmatrix}
      1 & -\frac{\mu_2}{m_2} \\
      0 & 1
    \end{pmatrix}
    \begin{pmatrix}
      m_2^{-\frac{1}{2}} & 0 \\
      0 & m_2^{\frac{1}{2}}
    \end{pmatrix}
  \end{equation}
  has norm at most $Y^{-\frac{1}{2}}$, then the bottom row of
  \begin{equation}
    \label{eq:cusp1}
    \begin{pmatrix}
      1 & 0 & 0 \\
      0 & a & b \\
      0 & c & d
    \end{pmatrix}
    \begin{pmatrix}
      1 & * & * \\
      0 & 1 & - \frac{\mu_2}{m_2} \\
      0 & 0 & 1
    \end{pmatrix}
    \begin{pmatrix}
      m_1^{-\frac{2}{3}}m_2^{-\frac{1}{3}}\ee^{2T} & 0 & 0 \\
      0 & m_1^{\frac{1}{3}}m_2^{-\frac{1}{3}}\ee^{-T} & 0 \\
      0 & 0 & m_1^{\frac{1}{3}} m_2^{\frac{2}{3}} \ee^{-T}
    \end{pmatrix}
  \end{equation}
  has norm at most $Y^{-\frac{1}{2}} \ee^{2T_0}$.
  Theorem \ref{theorem:rootsintersections} then implies the lemma.
\end{proof}

\begin{lemma}
  \label{lemma:setincusp}
  The number of $\xi \in I_l^{-1}$ with $\e^{6(T-T_0)}< N(\xi) \leq \ee^{6T}$ such that $g_l a(\xi) g_0^{-1} a(-T) \not\in \mcF(Y)$ is $O((Y^{-\frac{3}{7}}e^{2T_0} + Y^{-\frac{1}{7}}) \ee^{6T})$. 
\end{lemma}

\begin{proof}
  From \cite[Lemma~13]{MR4467125}, the equation
  \begin{equation}
    \label{eq:xibijection}
    g_l \tilde a(\xi) g_0^{-1} = \frac{\kappa_l}{N(\xi)^{\frac{1}{3}}} 
    \begin{pmatrix}
      q & * & * \\
      r_1 & * & * \\
      r_2 & * & * 
    \end{pmatrix}
  \end{equation}
  defines a bijection between $\xi \in I_l^{-1}$ and $(q, r_1, r_2) \in \ZZ^3$ under which primitive $\xi$ ($\xi I_l$ not divisible by rational integer) are mapped to primitive $(q, r_1, r_2)$.
  Here $\kappa_l$ is a constant related to the discriminant of the ideal $I_l$.

  From lemma \ref{lemma:badLU} and an elementary lattice point count (note that $q$ is a linear function of $\xi$), it follows that the number of $\xi$ with $N(\xi) \leq \e^{6T}$ and $q \leq \epsilon \e^{2T}$ is $O(\epsilon e^{6T}$.
  
  Performing an LU decomposition, we have
  \begin{equation}
    g_l a(\xi) g_0^{-1}a(-T) =  n(\tfrac{r_1}{q}, \tfrac{r_2}{q}) a(\tfrac{1}{2}\log (\tfrac{q}{N(\xi)^{\frac{1}{3}}}) -T) g
  \end{equation}
  with $|g| \ll \epsilon^{-1}$.
  Now if the there is $\gamma \in \Gamma$ so that $|
  \begin{pmatrix}
    1 & 0 & 0 
  \end{pmatrix}
  \gamma g_l \tilde a(\xi) g_0^{-1}| \leq Y^{-\frac{1}{2}}$, then
  \begin{equation}
    \label{eq:fareyheight}|
    \begin{pmatrix}
      1 & 0 & 0 
    \end{pmatrix}
    \gamma n(\tfrac{r_1}{q}, \tfrac{r_2}{q}) a(-T)| \ll \epsilon^{-2}(\epsilon^{-1} \e^{T_0}) Y^{-\frac{1}{2}}.
  \end{equation}

  It now suffices to upper bound the number of Farey points of denominator $q \leq \ee^{2T}$ such that there are coprime integers $c,d_1,d_2$ satisfying
  \begin{equation}
    \label{eq:cusp1condition}
    (c + d_1 \frac{r_1}{q} + d_2 \frac{r_2}{q} )^2 \ee^{4T} + d_1^2 \ee^{-2T} + d_2^2 \ee^{-2T}  \ll \epsilon^{-4}(\epsilon^{-2} + e^{2T_0}) \tfrac{1}{Y} =: \tfrac{1}{Y_1}. 
  \end{equation}
  This number is therefore bounded by the number of Farey points satisfying
  \begin{equation}
    \label{eq:cusp1condition1}
    |c + d_1 \frac{r_1}{q} + d_2 \frac{r_2}{q}| \ll Y_1^{-\frac{1}{2}} \ee^{-2T}
  \end{equation}
  for some coprime integers $c, d_1, d_2$ satisfying $d_1, d_2 \leq Y_1^{-\frac{1}{2}} \ee^{T}$.

  The quantity on the left of (\ref{eq:cusp1condition1}) is either $0$ or at least $\frac{1}{q}$, so if (\ref{eq:cusp1condition1}) is satisfied with $q \leq \ee^{2T}$ and $Y_1$ sufficiently small, then the Farey point $(\frac{r_1}{q}, \frac{r_2}{q})$ must lie on the line $c + d_1 x_1 + d_2 x_2  = 0$.
  Any other Farey point $(\frac{r_1'}{q'}, \frac{r_2'}{q'})$ on this line satisfies
  \begin{equation}
    \label{eq:spacingresult}
    \begin{pmatrix}
      r_1 r_2' - r_1' r_2 & r_2 q' - r_2' q & r_1'q - r_1 q'
    \end{pmatrix}
    = k
    \begin{pmatrix}
      c & d_1 & d_2
    \end{pmatrix}
  \end{equation}
  for a nonzero integer $k$.
  The distance between these Farey points is therefore at least $\frac{\sqrt{d_1^2+ d_2^2}}{q q'} \geq \ee^{-4T} \sqrt{d_1^2 + d_2^2}$.

  The total length of a given line $c + d_1 x_1 + d_2 x_2 = 0$ on the torus $\TT^2$ is $\frac{ \sqrt{d_1^2 + d_2^2}}{\gcd (d_1, d_2)}$, and the number of such lines for given $d_1$ and $d_2$ is $\gcd(d_1, d_2)$.
  Combined with the lower bound on the spacing between Farey points, we have that the total number of Farey points with denominators at most $\ee^{2T}$ on lines $c + d_1 x_1 + d_2 x_2 = 0$ with $d_1^2 + d_2^2 \ll Y_1^{-1}\ee^{2T}$ is $O(Y_1^{-1} \ee^{6T})$.
  The lemma now follows by choosing $\epsilon = Y^{-\frac{1}{7}}$. 
\end{proof}

Combining lemmas \ref{lemma:smalltobigcusp} and \ref{lemma:setincusp} shows that removing the $\mcF_1(Y)$ condition in \ref{prop:truncatedequidistribution} creates an arbitrary small remainder if $Y$ is small enough (depending on $T_0$).
The following lemma is used to remove the condition $N(\xi) \geq \e^{6(T-T_0)}$. 
Its proof, which we omit, is an elementary exercise in lattice point counting.
\begin{lemma}
  \label{lemma:idealcount}
  The number of $\xi \in \mcD \cap I_l^{-1}$ with $N(\xi) \leq \ee^{6(T - T_0)}$ is $O(\e^{6(T-T_0)})$.
\end{lemma}

\bibliographystyle{plain}
\bibliography{references} 

\begin{thebibliography}{10}

\bibitem{MR3330337}
Jayadev~S. Athreya, Jon Chaika, and Samuel Leli\`evre.
\newblock The gap distribution of slopes on the golden {L}.
\newblock In {\em Recent trends in ergodic theory and dynamical systems},
  volume 631 of {\em Contemp. Math.}, pages 47--62. Amer. Math. Soc.,
  Providence, RI, 2015.

\bibitem{MR3214280}
Jayadev~S. Athreya and Yitwah Cheung.
\newblock A {P}oincar\'e{} section for the horocycle flow on the space of
  lattices.
\newblock {\em Int. Math. Res. Not. IMRN}, (10):2643--2690, 2014.

\bibitem{MR1837099}
Florin~P. Boca, Cristian Cobeli, and Alexandru Zaharescu.
\newblock A conjecture of {R}. {R}.\ {H}all on {F}arey points.
\newblock {\em J. Reine Angew. Math.}, 535:207--236, 2001.

\bibitem{MR2058609}
Laurent Clozel and Emmanuel Ullmo.
\newblock \'equidistribution des points de {H}ecke.
\newblock In {\em Contributions to automorphic forms, geometry, and number
  theory}, pages 193--254. Johns Hopkins Univ. Press, Baltimore, MD, 2004.

\bibitem{Dartyge2015}
C\'{e}cile Dartyge.
\newblock Le probl\`eme de {T}ch\'{e}bychev pour le douzi\`eme polyn\^{o}me
  cyclotomique.
\newblock {\em Proc. Lond. Math. Soc. (3)}, 111(1):1--62, 2015.

\bibitem{delaBreteche2015}
R.~de~la Bret\`eche.
\newblock Plus grand facteur premier de valeurs de polyn\^{o}mes aux entiers.
\newblock {\em Acta Arith.}, 169(3):221--250, 2015.
\newblock With an appendix by de la Bret\`eche and J.-F. Mestre.

\bibitem{DukeFriedlanderIwaniec1995}
W.~Duke, J.~B. Friedlander, and H.~Iwaniec.
\newblock Equidistribution of roots of a quadratic congruence to prime moduli.
\newblock {\em Ann. of Math. (2)}, 141(2):423--441, 1995.

\bibitem{MR2515103}
Manfred Einsiedler, Elon Lindenstrauss, Philippe Michel, and Akshay Venkatesh.
\newblock Distribution of periodic torus orbits on homogeneous spaces.
\newblock {\em Duke Math. J.}, 148(1):119--174, 2009.

\bibitem{MR2776363}
Manfred Einsiedler, Elon Lindenstrauss, Philippe Michel, and Akshay Venkatesh.
\newblock Distribution of periodic torus orbits and {D}uke's theorem for cubic
  fields.
\newblock {\em Ann. of Math. (2)}, 173(2):815--885, 2011.

\bibitem{MR3484111}
Manfred Einsiedler, Shahar Mozes, Nimish Shah, and Uri Shapira.
\newblock Equidistribution of primitive rational points on expanding
  horospheres.
\newblock {\em Compos. Math.}, 152(4):667--692, 2016.

\bibitem{elbaz2023effectiveequidistributionprimitiverational}
Daniel El-Baz, Min Lee, and Andreas Strömbergsson.
\newblock Effective equidistribution of primitive rational points on expanding
  horospheres, 2023.

\bibitem{Heath-Brown2000}
D.~R. Heath-Brown.
\newblock The largest prime factor of $x^3 - 2$.
\newblock {\em Proceedings of the London Mathematical Society},
  82(3):554–596, 2000.

\bibitem{Hooley1964}
C.~Hooley.
\newblock On the distribution of the roots of polynomial congruences.
\newblock {\em Mathematika}, 11:39--49, 1964.

\bibitem{MR1359098}
D.~Y. Kleinbock and G.~A. Margulis.
\newblock Bounded orbits of nonquasiunipotent flows on homogeneous spaces.
\newblock In {\em Sina\u i's {M}oscow {S}eminar on {D}ynamical {S}ystems},
  volume 171 of {\em Amer. Math. Soc. Transl. Ser. 2}, pages 141--172. Amer.
  Math. Soc., Providence, RI, 1996.

\bibitem{MR1920389}
Anthony~W. Knapp.
\newblock {\em Lie groups beyond an introduction}, volume 140 of {\em Progress
  in Mathematics}.
\newblock Birkh\"auser Boston, Inc., Boston, MA, second edition, 2002.

\bibitem{KowalskiSoundararajan2020}
E.~Kowalski and K.~Soundararajan.
\newblock Equidistribution from the {C}hinese remainder theorem.
\newblock {\em Adv. Math.}, 385:Paper No. 107776, 36, 2021.

\bibitem{MR3873538}
Min Lee and Jens Marklof.
\newblock Effective equidistribution of rational points on expanding
  horospheres.
\newblock {\em Int. Math. Res. Not. IMRN}, 21:6581--6610, 2018.

\bibitem{MR4795423}
Zonglin Li and Matthew Welsh.
\newblock Geometry and distribution of roots of {$\mu ^2\equiv D\pmod m$} with
  {$D \equiv 1\ \pmod4$}.
\newblock {\em Acta Arith.}, 215(3):193--228, 2024.

\bibitem{MR238801}
Yu.\~V. Linnik.
\newblock {\em Ergodic properties of algebraic fields}, volume Band 45 of {\em
  Ergebnisse der Mathematik und ihrer Grenzgebiete [Results in Mathematics and
  Related Areas]}.
\newblock Springer-Verlag New York, Inc., New York, 1968.
\newblock Translated from the Russian by M. S. Keane.

\bibitem{MR2035655}
Grigoriy~A. Margulis.
\newblock {\em On some aspects of the theory of {A}nosov systems}.
\newblock Springer Monographs in Mathematics. Springer-Verlag, Berlin, 2004.
\newblock With a survey by Richard Sharp: Periodic orbits of hyperbolic flows,
  Translated from the Russian by Valentina Vladimirovna Szulikowska.

\bibitem{MR3079137}
Jens Marklof.
\newblock Fine-scale statistics for the multidimensional {F}arey sequence.
\newblock In {\em Limit theorems in probability, statistics and number theory},
  volume~42 of {\em Springer Proc. Math. Stat.}, pages 49--57. Springer,
  Heidelberg, 2013.

\bibitem{MR2726104}
Jens Marklof and Andreas Str\"ombergsson.
\newblock The distribution of free path lengths in the periodic {L}orentz gas
  and related lattice point problems.
\newblock {\em Ann. of Math. (2)}, 172(3):1949--2033, 2010.

\bibitem{MR4654052}
Jens Marklof and Matthew Welsh.
\newblock Fine-scale distribution of roots of quadratic congruences.
\newblock {\em Duke Math. J.}, 172(12):2303--2364, 2023.

\bibitem{MR3496932}
Audrey Terras.
\newblock {\em Harmonic analysis on symmetric spaces---higher rank spaces,
  positive definite matrix space and generalizations}.
\newblock Springer, New York, second edition, 2016.

\bibitem{MR3567252}
Caglar Uyanik and Grace Work.
\newblock The distribution of gaps for saddle connections on the octagon.
\newblock {\em Int. Math. Res. Not. IMRN}, (18):5569--5602, 2016.

\bibitem{MR4467125}
Matthew Welsh.
\newblock Parametrizing roots of polynomial congruences.
\newblock {\em Algebra Number Theory}, 16(4):881--918, 2022.

\end{thebibliography}

\end{document}